\documentclass[10pt]{article}

\usepackage{amsmath} 
\usepackage{amsthm} 
\usepackage{amssymb} 
\usepackage{graphicx} 
\usepackage{nicefrac} 
\usepackage{longtable} 
\usepackage{textcomp} 

\usepackage{paralist}

\usepackage[small,nohug,heads=vee]{diagrams}
\diagramstyle[labelstyle=\scriptstyle]

\usepackage[table]{xcolor}
\usepackage{xy}
\usepackage{tikz}

\definecolor{cof}{RGB}{219,144,71}
\definecolor{pur}{RGB}{186,146,162}
\definecolor{greeo}{RGB}{91,173,69}
\definecolor{greet}{RGB}{52,111,72}

\usepackage{array}
\newcolumntype{L}[1]{>{\raggedright\let\newline\\\arraybackslash\hspace{0pt}}m{#1}}
\newcolumntype{C}[1]{>{\centering\let\newline\\\arraybackslash\hspace{0pt}}m{#1}}
\newcolumntype{R}[1]{>{\raggedleft\let\newline\\\arraybackslash\hspace{0pt}}m{#1}}

\newtheorem{teo}{Theorem}[section]
\newtheorem{lem}[teo]{Lemma}
\newtheorem{pro}[teo]{Proposition}

\newtheorem{defi}[teo]{Definition}

\newtheorem*{notation}{Notation}
\newtheorem*{claim}{Claim}

\newtheorem*{conj}{Conjecture}
\newtheorem*{Corollary}{Corollary}

\newcommand{\rat}{\mathbb{Q}}
\newcommand{\nat}{\mathbb{N}}
\newcommand{\com}{\mathbb{C}}
\newcommand{\inte}{\mathbb{Z}}

\newcommand{\field}{\mathbf{k}}
\newcommand{\pfield}{\mathbb{F}_p}

\newcommand{\Zp}{\nicefrac{\inte}{p\inte}}

\newcommand{\Gl}[2]{\operatorname{GL}_{#1}\left(#2\right)}
\newcommand{\pseudo}[1]{\operatorname{Pseudo}\left(#1\right)}

\newcommand{\Ab}{\mathbf{Ab}}

\newcommand{\Var}[1]{\mathbf{Var_{#1}}}
\newcommand{\Stck}[1]{\mathbf{Stack_{#1}}}

\newcommand{\hei}[1]{H_{#1}}
\newcommand{\var}[1]{\nicefrac{\mathbb{P}^{#1-1}}{A_{#1}}}
\newcommand{\varns}[1]{X_{#1}}
\newcommand{\open}[1]{U_{#1}}
\newcommand{\centro}{\operatorname{Z}\hei{p}}
\newcommand{\resolutionname}{f}
\newcommand{\representationname}{\rho}
\newcommand{\quotientname}{\pi}
\newcommand{\res}{\stackrel{\resolutionname}{\rightarrow}}
\newcommand{\repr}{\stackrel{\representationname}{\rightarrow}}

\newcommand{\Stab}[2]{\operatorname{Stab}_{#1}(#2)}
\newcommand{\barB}{\widehat{B}}

\newcommand{\Sp}{S_p}
\newcommand{\Singu}[1]{\operatorname{\operatorname{Sing}}\left(#1\right)}

\newcommand{\Hom}[3]{\operatorname{H}^{#1}\left(#2; #3\right)}
\newcommand{\Hrel}[4]{\operatorname{H}^{#1}_{#3}\left(#2; #4\right)}
\newcommand{\Homred}[2]{\operatorname{H}^{#1}\left(#2\right)}

\newcommand{\grot}[1]{K_0(#1)}
\newcommand{\grotcom}[1]{\widehat{K_0}(#1)}

\newcommand{\Lo}[1]{L_0(#1)}

\newcommand{\Lclass}{\mathbb{L}}

\newcommand{\eke}[2]{\operatorname{e}_{#1}\left(#2\right)}

\newcommand{\B}[1]{\operatorname{\mathcal{B}} #1}
\newcommand{\cl}{\{\B{G}\}}

\newcommand{\M}{\cellcolor[gray]{0.8}}

\title{The Ekedahl Invariants for finite groups}
\author{Ivan Martino}

\begin{document}
\maketitle

\begin{abstract}
  In 2009 Ekedahl introduced certain cohomological invariants of finite groups
  which are naturally related to the Noether Problem.
  We show that these invariants are trivial for every finite group in $\Gl{3}{\field}$ and for the fifth discrete Heisenberg group $\hei{5}$.
  Moreover in the case of finite linear groups with abelian projective reduction, these invariants satisfy a recurrence relation in a certain Grothendieck group for abelian groups. 
\end{abstract}

%
%




Let $V$ be a finite dimension faithful linear representation of a finite group $G$ over a field $\field$ of characteristic prime to the order of $G$.
Inspired by a work of Bergstr\"om \cite{Bergstrom-pointcount}, Ekedahl in \cite{Ekedahl-inv} and \cite{EkedahlStack} investigated a motivic version of point counting over finite fields.
One application of Ekedahl's results is to study when the equality 
\begin{equation}\label{eq-equality}
    \{\nicefrac{\operatorname{GL}(V)}{G}\}=\{\operatorname{GL}(V)\}
\end{equation}
holds in the Kontsevich value ring $\grotcom{\Var{\field}}$ of algebraic $\field$-varieties.
%

All the known cases, where this equality fails are counterexamples to the Noether Problem.
In the beginning of the last century, Noether \cite{Noether1917} wondered about the rationality of the field extension $\field(V)^G/\field$ for any finite group $G$ and any field $\field$, where $\field(V)^G$ are the invariants of the field of rational functions $\field(V)$ over the regular representation $V$ of $G$. (The Noether Problem can be stated for any arbitrary field, but we will not need the full generality.)

\noindent
The first counterexample, $\rat(V)^{\nicefrac{\inte}{47\inte}}/\rat$, was given by Swan in \cite{Swan1969} and it appeared during 1969.

\noindent
In the 1980s more counterexamples were found: for every prime $p$ Saltman \cite{Saltman1984} and Bogomolov \cite{Bogomolov1988} showed that there exists a group of order $p^9$ and, respectively, of order $p^6$ such that the extension $\com(V)^{G}/\com$ is not rational.  

\noindent
Saltman used the second unramified cohomology group of the field $\com(V)^{G}$, $\operatorname{H}_{nr}^{2}(\com(V)^{G}, \nicefrac{\rat}{\inte})$, as a cohomological obstruction to rationality.
Later, Bogomolov found a group cohomology expression for $\operatorname{H}_{nr}^{2}(\com(V)^{G}, \nicefrac{\rat}{\inte})$ which now bears his name and is denoted by $B_0(G)$.

\noindent
Recently Hoshi, Kang and Kunyavskii investigateed the case where $|G|=p^5$.
They showed that $B_0(G)\neq 0$ if and only if $G$ belongs to the isoclinism family $\phi_{10}$; see \cite{HKM-noether}. 
%


In 2009, Ekedahl \cite{EkedahlStack} defines, for every integer $k$, a \textit{cohomological} map
\[
    \operatorname{\mathcal{H}}^k: \grotcom{\Var{\field}}\rightarrow \Lo{\Ab},
\]
where $\Lo{\Ab}$ is the group generated by the isomorphism classes 
$\{G\}$ of finitely generated abelian groups $G$ under the relation $\{A\oplus B\}=\{A\}+\{B\}$. 

\noindent
Let $\mathbb{L}^i$ be the class of the affine space $\mathbb{A}^i_{\field}$ in $\grotcom{\Var{\field}}$. In particular, $\mathbb{L}^0$ is the class of a point $\{*\}=\{\operatorname{Spec}(\field)\}$.
We observe that $\mathbb{L}^i$ is invertible in $\grotcom{\Var{\field}}$.
To define $\operatorname{\mathcal{H}}^k$ on $\grotcom{\Var{\field}}$ is enough to set $\operatorname{\mathcal{H}}^k(\nicefrac{\{X\}}{\mathbb{L}^m})=\{\Hom{k+2m}{X}{\mathbb{Z}}\}$ for every smooth and proper $\field$-variety $X$ (for more details see Section 3 in \cite{Intro-Ekedahl-Invariants}).

The class $\{\B{G}\}$ of the classifying stack of $G$ is an element of $\grotcom{\Var{\field}}$ (see Proposition 2.5.b in \cite{Intro-Ekedahl-Invariants}) and so one can define:
\setcounter{section}{1}\setcounter{teo}{1}
\begin{defi}
  For every integer $i$, the $i$-th \emph{Ekedahl invariant} $\eke{i}{G}$ of the group $G$ is $\operatorname{\mathcal{H}}^{-i}(\{\B{G}\})$ in $\Lo{\Ab}$.
  We say that the Ekedahl invariants of $G$ are \emph{trivial} if $\eke{i}{G}=0$ for all integer $i\neq 0$.
\end{defi}

\noindent
In Proposition 2.5.a of \cite{Intro-Ekedahl-Invariants}, the author rephrases the equality (\ref{eq-equality}) in terms of algebraic stacks, using the expression
\[
  \{\B{G}\}=\frac{\{\nicefrac{\operatorname{GL}(V)}{G}\}}{\{\operatorname{GL}(V)\}}\in\grotcom{\Var{\field}}.
\]
%
Since $\{\operatorname{GL}(V)\}$ is invertible in $\grotcom{\Var{\field}}$, the equation (\ref{eq-equality}) holds if and only if $\{\B{G}\}=\{*\}$ and, if this is the case, then the Ekedahl invariants of $G$ are \emph{trivial}, because $\operatorname{\mathcal{H}}^0(\{*\})=\{\inte\}$ and $\operatorname{\mathcal{H}}^k(\{*\})=0$ for $k\neq 0$.

These new invariants seem a natural generalization of the Bogomolov multiplier $B_0(G)$ because of the following result.
\begin{teo}[Thm 5.1 of \cite{Ekedahl-inv}]\label{thm-ekedahl}
  Assume $\operatorname{char}(\field)=0$.
  If $G$ is a finite group, then $\eke{i}{G}=0$ for every $i<0$, $\eke{0}{G}=\{\mathbb{Z}\}$, $\eke{1}{G}=0$ and $\eke{2}{G}=\{B_0(G)^{\vee}\}$, where $B_0(G)^{\vee}$ is the pontryagin dual of the Bogomolov multiplier of the group $G$.
  
  \noindent
  Moreover, for $i>0$, the invariant $e_i(G)$ is an integer linear combination of classes of finite abelian groups.
\end{teo}

\noindent

Using that $\eke{2}{G}=\{B_0(G)^{\vee}\}$, one finds some groups with non-trivial Ekedahl invariants (and so $\{\B{G}\}\neq \{*\}$).

\begin{Corollary}
If $G$ is one of the group of order $p^9$ defined in \cite{Saltman1984}, of order $p^6$ defined in \cite{Bogomolov1988} or of order $p^5$ belonging to the isoclinism family $\phi_{10}$ (see \cite{HKM-noether}), then the second Ekedahl invariant is non-zero and so $\{\B{G}\}\neq \{*\}$.
\end{Corollary}

\noindent
It is not clear if higher Ekedahl invariants are obstructions to the rationality of the extension $\field(V)^G/\field$.


%

In Corollary 5.8 of \cite{Ekedahl-inv}, it is also proved that $\{ \B{\nicefrac{\inte}{47\inte}} \}\neq \{*\}\in \grotcom{\Var{\rat}}$, but it is unknown if $\{\B{G}\}\neq \{*\}$ implies a negative answer to the Noether problem.

To the author's knowledge, there are no examples of finite group $G$ such that $B_0(G)=0$ (i.e. $\eke{2}{G}=0$) and $\eke{3}{G}\neq 0$. 
It is worth mentioning that Peyre \cite{Peyre-3-th-unramified} showed a class of finite groups having $B_0(G)=0$, but non-trivial third unramified cohomology: it is not clear to the author if this is connected to the non-triviality of higher Ekedahl invariants.

There are a lot of classes of finite groups that have trivial Ekedahl invariants (see Section 4 of \cite{Intro-Ekedahl-Invariants}).
We want to mention one example relevant for this work: if $G\subset GL_1(\field)$, then $\{\B{G}\}=\{*\}\in \grotcom{\Var{\field}}$, under the condition that $\field$ contains a primitive root of unity of degree the exponent of $G$ \cite[Proposition 3.2]{Ekedahl-inv}.
Since, every abelian group $A$ is a direct product of cyclic groups, $\{\B{\mu_n}\}=\{*\}$ implies that $\{\B{A}\}=\{*\}$: this will be useful in the proof of Theorem \ref{thm-abelian-P-H-X}.

\vspace{0.3cm}

In this paper we first generalize the just mention result.

%
\setcounter{section}{2}\setcounter{teo}{3}
\begin{teo}
  Assume $\operatorname{char}(\field)=0$ and assume that $\field$ contains a primitive root of unity of degree the exponent of $G$.
  If $G$ is a finite subgroup of $\Gl{3}{\field}$, then $\cl=\{*\}$ in $\grotcom{\Var{\field}}$ and the Ekedahl invariants of $G$ are trivial.
\end{teo}

\noindent
The case when $G$ is a finite subgroup of $\Gl{4}{\field}$ is more complicated because it involves the study of the class in $\grotcom{\Var{\field}}$ of resolution of singularities of the affine varieties $\nicefrac{\mathbb{A}^3}{A}$, for a finite group $A\subset \Gl{3}{\field}$.

Nevertheless, we can still tackle this problem if we assume the ground field to be the complex field (and we use the toric varieties tools). 
Indeed, in higher dimension we are able to prove that the Ekedahl invariants satisfy a kind of recurrence equation in $\Lo{\Ab}$.

\setcounter{section}{3}\setcounter{teo}{1}
\begin{teo}
Let $G$ be a finite subgroup of $\Gl{n}{\com}$ and let $H$ be the image of $G$ under the canonical projection $\Gl{n}{\com}\rightarrow \operatorname{PGL}_n(\com)$.

  If $H$ is abelian and if $\nicefrac{\mathbb{P}^{n-1}_{\mathbb{C}}}{H}$ has only zero dimensional singularities, then each singularity is toroidal and for every integer $k$
  \begin{equation*}
    \eke{k}{G}+\eke{k+2}{G}+\dots+\eke{k+2(n-1)}{G}=\{\Hom{-k}{X}{\inte}\},
  \end{equation*}
  where $X\rightarrow\nicefrac{\mathbb{P}^{n-1}}{H}$ is a smooth and proper resolution with normal crossing toric exceptional divisors.
\end{teo}

We focus, then, on the $p$-discrete Heisenberg group $\hei{p}$, where we only deal with cyclic quotient singularities. 
This is an interesting candidate for the study of the Ekedahl invariants, because $B_0(\hei{p})=0$ (using Lemma 4.9 in \cite{Bogomolov1988}) and so the first unknown Ekedahl invariant is $\eke{3}{\hei{p}}$. 

\setcounter{section}{4}\setcounter{teo}{3}
\begin{teo}
  The Ekedahl invariants of the fifth discrete Heisenberg group $\hei{5}$ are trivial.
\end{teo}

\noindent
Kang has proved that the Noether problem for the Heisenberg group $\hei{p}$ has a positive answer. Indeed using Theorem 1.9 of \cite{Kang-Rationality} one shows that the extension $\com(x_g, g\in \hei{p})^{\hei{p}}/\com$ is rational.

We show a general approach for the study of the Ekedahl invariants of $\hei{p}$, but we narrow down our investigation to $p=5$ because of the difficulties to extend the technical result in Theorem \ref{thm-rational-cohomology-exceptional-divisor-articolo}.
Therefore, it is natural to conjecture that:
\begin{conj}
  The Ekedahl invariants of the Heisenberg group $H_p$ of order $p^3$ are trivial.
\end{conj}

After a preliminary section where we review the theory of the Ekedahl invariants, in 
Section \ref{sec:GL23} we prove that these are trivial for all finite subgroups of $\Gl{3}{\field}$.
Then, in Section \ref{sec-abelian-subgroups}, we study the finite linear groups with abelian projective quotient and in the last section we deal with the fifth Heisenberg group. 

\begin{notation}\em
  Along the whole work, we consider finite linear groups.
  In Section \ref{sec:GL23} we also assume that $\field$ contains a primitive root of unity of degree the exponent of $G$.
  In Section \ref{sec-abelian-subgroups} and Section \ref{sec:h-5} the ground field is $\com$.

  We set $*=\operatorname{Spec}(\field)$.
  
  Finally, if $X$ is a scheme with a $G$-action, then we denote by $\nicefrac{X}{G}$ the schematic quotient and by $[\nicefrac{X}{G}]$ the stack quotient. 
  %
\end{notation}

\setcounter{section}{0}\setcounter{teo}{0}
\section{Preliminaries}\label{sec-preliminary}
In this section we review the background and the definition of the Ekedahl invariants. 
A more complete and self contained introduction to these new invariants can be found in \cite{Intro-Ekedahl-Invariants}.
In this section we consider finite linear group and we work over a field $\field$ of characteristic prime to the order of $G$.

The Grothendieck ring of algebraic varieties $\grot{\Var{\field}}$ is the group generated by the isomorphism classes $\{X\}$ of algebraic $\field$-varieties $X$, subject to the relation $\{X\}=\{Z\}+\{X\setminus Z\}$, for all closed subvarieties $Z$ of $X$. The group $\grot{\Var{\field}}$ has a ring structure given by $\{X\}\cdot \{Y\}=\{X\times Y\}$.
Let $\mathbb{L}$ be the class of the affine line.
The completion of $\grot{\Var{\field}}[\mathbb{L}^{-1}]$ with respect to the dimension filtration 
\[
  Fil^{n}\left(\grot{\Var{\field}}[\mathbb{L}^{-1}]\right)=\{\nicefrac{\{X\}}{\mathbb{L}^i}: \dim{X}-i\leq n\}
\]
is called the Kontsevich value ring and denoted by $\grotcom{\Var{\field}}$.


In \cite{EkedahlStack}, Ekedahl introduced the concept of Grothendieck group of algebraic stacks.
\begin{defi}\label{defi-grot-stack}
  We denote by $\grot{\Stck{\field}}$ the Grothendieck group of algebraic $\field$-stacks. This is the group generated by the isomorphism classes $\{X\}$ of algebraic $\field$-stacks $X$ of finite type all of whose automorphism group scheme are affine 
  (algebraic $\field$-stacks of finite type with affine stabilizers, in short). 
  %
  The elements of this group fulfill the following relations:
  \begin{enumerate}
    \item for each closed substack $Y$ of $X$, $\{X\}=\{Y\}+\{Z\}$, where $Z$ is the complement of $Y$ in $X$;
    \item for each vector bundle $E$ of constant rank $n$ over $X$, $\{E\}=\{X\times \mathbb{A}^n\}$.
  \end{enumerate}
\end{defi}

\noindent
Similarly to $\grot{\Var{\field}}$, $\grot{\Stck{\field}}$ has a ring structure.
In Theorem 4.1 of \cite{EkedahlStack} it is proved that $\grot{\Stck{\field}}=\grot{\Var{\field}}[\Lclass^{-1}, (\Lclass^{n}-1)^{-1}, \forall n\in \nat]$. 
Moreover we observe in Lemma 2.2 of \cite{Intro-Ekedahl-Invariants} the completion map $\grot{\Var{\field}}[\Lclass^{-1}] \rightarrow \grotcom{\Var{\field}}$ factors through 
\[
  \grot{\Var{\field}}[\Lclass^{-1}] \rightarrow \grot{\Stck{\field}}\rightarrow \grotcom{\Var{\field}}.
\]
The classifying stack of the group $G$ is usually defined as the stack quotient $\B{G}=[\nicefrac{*}{G}]$ and, via this map, one sees the class of the classifying stack $\cl$ as an element of $\grotcom{\Var{\field}}$ (see Proposition 2.5.b in \cite{Intro-Ekedahl-Invariants}).

Using the Bittner presentation (see \cite{Bittner2004}), given a integer $k$, Ekedahl defines in \cite{EkedahlStack} a \textit{cohomological} map for the Kontsevich value ring, sending
$\nicefrac{\{X\}}{\mathbb{L}^m}$ to $\{\Hom{k+2m}{X}{\mathbb{Z}}\}$, for every smooth and proper $\field$-variety $X$:
\begin{eqnarray*}
  \operatorname{\mathcal{H}}^k: \grotcom{\Var{\field}}&\rightarrow& \Lo{\Ab}\\
				\nicefrac{\{X\}}{\mathbb{L}^m}&\mapsto  &\{\Hom{k+2m}{X}{\mathbb{Z}}\}.
\end{eqnarray*}
If $\field=\com$, it is natural to assign to every smooth and proper $\field$-variety $X$ the class of its integral cohomology group $\Hom{k}{X}{\inte}$.
If instead $\field$ is different from $\com$, then we send $\{X\}$ to the class $\{\Hom{k}{X}{\inte}\}$ defined as $\operatorname{dim}\Hom{k}{X}{\inte}\{\inte\}+\sum_p\{\operatorname{tor}\Hom{k}{X}{\mathbb{Z}_p}\}$.
This map can be extended to $\grotcom{\Var{\field}}$.
In Section 3 of \cite{Intro-Ekedahl-Invariants}, we prove that this map is well defined.

\begin{notation}\em
  Every cohomology group (if not explicitly expressed differently) is the singular cohomology group with integer coefficients, that is $\Homred{k}{-}=\Hom{k}{-}{\inte}$.
  %
\end{notation}

\begin{defi}\label{def-ekedahl-invariants}
  For every integer $i$, the $i$-th \emph{Ekedahl invariant} $\eke{i}{G}$ of the group $G$ is $\operatorname{\mathcal{H}}^{-i}(\{\B{G}\})$ in $\Lo{\Ab}$.
  We say that the Ekedahl invariants of $G$ are \emph{trivial} if $\eke{i}{G}=0$ for all integers $i\neq 0$.
\end{defi}

\noindent

\section{The finite subgroups of $\Gl{3}{\field}$}\label{sec:GL23}
In this section we assume that $\operatorname{char}(\field)$ is prime with the order of $G$ and that $\field$ contains a primitive root of unity of degree the exponent of $G$. 
Let $V$ be an $n$-dimensional $\field$-vector space.
%
%
%
%
We assume $G$ to be a finite subgroup $\operatorname{GL}_n(\field)$ and $H$ be its quotient in $\operatorname{PGL}_n(\field)$:
  \begin{diagram}
  0	&\rTo   &K		&\rTo	&G       	&\rTo    	& H		&\rTo	&0\\
	  &   	&\dInto	&	&\dInto  	&       	& \dTo		&		&\\
  0	&\rTo   &\mathbb{G}_m	&\rTo	&\operatorname{GL}_n(\field)    	&\rTo   	& \operatorname{PGL}_n(\field)		&\rTo	&0.
  \end{diagram}

%
%
\noindent
We sometimes use for simplicity $\mathbb{P}^{n-1}$ for the projective space $\mathbb{P}(V)$.


To get information about $\cl$, we study the class of the stack quotient $\{[\nicefrac{\mathbb{P}(V)}{G}]\}$.

\begin{lem}[Lemma 2.4 of \cite{Intro-Ekedahl-Invariants}]\label{lem-utile}
$\cl\left(1+\Lclass+\dots+\Lclass^{n}\right)=\{[\nicefrac{\mathbb{P}(V)}{G}]\}$ in $\grotcom{\Var{\field}}$.
\end{lem}

By definition $H$ acts on $\mathbb{P}(V)$ and $\nicefrac{\mathbb{P}(V)}{G}\cong \nicefrac{\mathbb{P}(V)}{H}$.
Their stack quotients are not isomorphic, but the classes of their stack quotients are equal in $\grotcom{\Var{\field}}$.

\begin{pro}\label{prop-equality-classes-PV}
	%
	%
	Let $\mathcal{L}\rightarrow X$ be a $G$-equivariant line bundle over a smooth scheme $X$, where the group $G$ acts faithfully on $\mathcal{L}$.
	Assume that the kernel $K$ of such action on $X$ is cyclic and set $H=\nicefrac{G}{K}$.
	Then $\{[\nicefrac{X}{G}]\}=\{[\nicefrac{X}{H}]\}\in \grotcom{\Var{\field}}$.
	
	As a consequence, if $G$ be a subgroup of $\operatorname{GL}(V)$ and $H$ is its quotient in $\operatorname{PGL}(V)$, then $\{[\nicefrac{\mathbb{P}(V)}{G}]\}=\{[\nicefrac{\mathbb{P}(V)}{H}]\}\in \grotcom{\Var{\field}}$.
\end{pro}
\begin{proof}

  The natural map $[\nicefrac{\mathcal{L}}{G}]\rightarrow [\nicefrac{X}{G}]$ is a $\field^*$-torsor and using Proposition 1.1.\textit{ii} in \cite{EkedahlStack} one obtains $\{[\nicefrac{\mathcal{L}}{G}]\}=(\mathbb{L}-1)\{[\nicefrac{X}{G}]\}$.
  Similarly from $[\nicefrac{(\nicefrac{\mathcal{L}}{K})}{H}]\rightarrow [\nicefrac{X}{H}]$, one gets $\{[\nicefrac{(\nicefrac{\mathcal{L}}{K})}{H}]\}=(\mathbb{L}-1)\{[\nicefrac{X}{H}]\}$.
  The first part of the statement follows from $[\nicefrac{(\nicefrac{\mathcal{L}}{K})}{H}]=[\nicefrac{L}{G}]$.
  
  The second part of the statement follows by setting $\mathcal{L}=V\setminus \{O\}$, $X=\mathbb{P}(V)$ with the natural tautological bundle.
\end{proof}

\noindent
We now set up notations, definitions and remarks regarding the quotient of algebraic varieties by finite groups. (This is necessary to introduce next lemma.)

\noindent
Let $Y$ be a smooth complex quasi-projective algebraic variety and let $A$ be a finite group of automorphisms of $Y$. 
%
Let $Y\rightarrow \nicefrac{Y}{A}$ be the quotient map and $\bar{y}$ be the image of $y$.

\noindent
A nontrivial element of $G\subset \operatorname{GL}(V)$ is a pseudo-reflection if it fixes pointwise a codimension one hyperplane in $V$. (In the literature pseudo-reflections are sometimes called reflections.)
The pseudo reflection subgroup, $\pseudo{G}$, of $G\subset \operatorname{GL}(V)$ is its subgroup generated by pseudo-reflections.
The Chevalley-Shephard-Todd Theorem says that the quotient $\nicefrac{V}{G}$ is smooth if and only if $G=\pseudo{G}$.
%

\noindent
The well known Cartan Lemma says that for all the points $y$ of $Y$, the action of the stabilizer $\Stab{y}{A}$ of $y$ on $Y$ induces an action of $\Stab{y}{A}$ on the tangent space on $y$, $T_yY$.
Moreover the analytic germ $(\nicefrac{Y}{A}, \bar{y})$ is isomorphic to $(\nicefrac{T_yY}{A}, \bar{O})$, where $\bar{O}$ is the image of the origin $O\in T_yY$ under the quotient map $T_yY\rightarrow \nicefrac{T_yY}{A}$.
An easy consequence is that for all the points $y$ of $Y$, $\Stab{y}{A}\subseteq \operatorname{GL}_{\operatorname{dim}(Y)}$ and one also proves that $p$ is a singular point of $\nicefrac{V}{G}$, $p\in \Singu{\nicefrac{V}{G}}$, if and only if $\pseudo{\Stab{p}{G}}\neq\Stab{p}{G}$.

%

Comparing the classes $\{[\nicefrac{\mathbb{P}(V)}{H}]\}$ and $\{\nicefrac{\mathbb{P}(V)}{H}\}$, we are going to prove that the Ekedahl invariants for every finite subgroup $G$ in $\Gl{3}{\field}$ are trivial.
We prove it by induction on $n$, where $G$ is a finite subgroup of $\Gl{n}{\field}$, and $n = 1, 2, 3$. 
The base case, $n = 1$, is covered by Proposition 3.2 in \cite{Ekedahl-inv}.
%

\begin{pro}\label{pro-GinGL2-trivial}
  If $G$ is a finite subgroup of $\Gl{2}{\field}$ then $\cl=\{*\}$ in $\grotcom{\Var{\field}}$ and the Ekedahl invariants of $G$ are trivial.
\end{pro}
\begin{proof}
  Let $U$ be the non empty open subset of $\mathbb{P}^1$ where $H$ acts freely.
  We set $I=\mathbb{P}^1\setminus U$, the finite set of non trivial stabilizer points.
  Using property 1) in Definition \ref{defi-grot-stack}, we write 
  \[
    \{[\nicefrac{\mathbb{P}^1}{H}]\}=\{\nicefrac{U}{H}\}+\sum_{p\in \nicefrac{I}{H}}\{[\nicefrac{p}{\Stab{p}{H}}]\}=\{\nicefrac{U}{H}\}+\sum_{p\in \nicefrac{I}{H}}\{\B{\Stab{p}{H}}\}.
  \]
  We observe that the sum is taken over the orbits of the non trivial stabilizer points.
  Similarly, in the classical quotient,
  $$\{\nicefrac{\mathbb{P}^1}{H}\}=\{\nicefrac{U}{H}\}+\sum_{p\in \nicefrac{I}{H}}\{*\}.$$
  
  \noindent
  Therefore, comparing the latter expressions: 
  $$\{[\nicefrac{\mathbb{P}^1}{H}]\}=\{\nicefrac{\mathbb{P}^1}{H}\}+\sum_{p}(\{\B{\Stab{p}{H}}\} -\{*\}).$$
  Using (in order) that $\{[\mathbb{P}^1/G]\}=\left(1+\Lclass\right)\cl$, Proposition \ref{prop-equality-classes-PV}, the previous formula 
  and $\nicefrac{\mathbb{P}^1}{H}\cong\mathbb{P}^1$, one has 
  \[
    \{\B{G}\}(1+\Lclass)=\{\mathbb{P}^1\} + \sum_{p}(\{\B{\Stab{p}{H}}\} -\{*\}).
  \]
  Using Cartan's Lemma, $\Stab{p}{H}$ is a subgroup of $\operatorname{GL}_1$ and, hence, for Proposition 3.2 in \cite{Ekedahl-inv}, $\{\B{\Stab{p}{H}}\}=\{*\}$ for every $p\in \nicefrac{I}{H}$.
  Hence, $\{\B{G}\}(1+\Lclass)=\{\mathbb{P}^1\}$ and this implies $\{\B{G}\}=1$, because $\Lclass^n-1$ is invertible in $\grotcom{\Var{\field}}$, $\Lclass^2-1=(\Lclass-1)(\Lclass+1)$ and so $1+\Lclass$ is invertible too.
\end{proof}

Note that we actually proved that $\{[\nicefrac{\mathbb{P}^1}{H}]\}=\{\mathbb{P}^1\}$ in $\grotcom{\Var{\field}}$ and in a similar way we tackle the next case. 
We are going to use resolution of singularities and for this we require the characteristic of the base field to be zero.

\begin{teo}\label{thm-GinGL3-trivial}
	Assume $\operatorname{char}(k)=0$.
    If $G$ is a finite subgroup of $\Gl{3}{\field}$ then $\cl=1$ in $\grotcom{\Var{\field}}$ and the Ekedahl invariants of $G$ are trivial. 
\end{teo}
\begin{proof}
  Using equation $\{[\mathbb{P}^2/G]\}=\left(1+\Lclass+\Lclass^{2}\right)\cl$ and Proposition \ref{prop-equality-classes-PV}, we know that $\{\B{G}\}\{\mathbb{P}^2\}=\{[\nicefrac{\mathbb{P}^2}{H}]\}$.
  Since $\{\mathbb{P}^2\}$ is invertible in $\grotcom{\Var{\field}}$, it is sufficient to prove that $\{[\nicefrac{\mathbb{P}^2}{H}]\}=\{\mathbb{P}^2\}$.
  
  Let $U$ be the open subset of $\mathbb{P}^2$ where $H$ acts freely and let $C$ be the complement of $U$ in $\mathbb{P}^2$.
  We denote by $C_0$ and $C_1$ respectively the dimension zero and the dimension one closed subsets of $C$ so that $C=C_0\sqcup C_1$.

  One observes that $[\nicefrac{C_0}{H}]$ is the disjoint union of a finite number of quotient stacks $[\nicefrac{O_i}{H}]$ where $O_i$ are the orbits of $P_i\in C_0$ under the action of $H$.
  We note that $[\nicefrac{O_i}{H}]=[\nicefrac{P_i}{\Stab{P_i}{H}}]=\B{\Stab{P_i}{H}}$.
  By Cartan's Lemma, $\Stab{P_i}{H}$ is a subgroup of $\Gl{2}{\field}$ and then, by using Proposition \ref{pro-GinGL2-trivial}, 
  $$\{[\nicefrac{O_i}{H}]\}=\{\B{\Stab{P_i}{H}}\}=\{\nicefrac{O_i}{H}\}=\{*\}.$$
  Therefore $\{[\nicefrac{C_0}{H}]\}=\{\nicefrac{C_0}{H}\}$.

  We observe that $\{[\nicefrac{S}{H}]\}=\{\nicefrac{S}{H}\}$ holds for every finite stable subset $S$ of $\mathbb{P}^2$ with the same argument.

  The set $C_1$ is the union of a finite number of lines $L_i$. 
  %
  We denote by $I$ the union of pairwise intersections $L_i \cap L_j$, for $i \neq j$.
  %
  We also denote by $C_1^*$ the complement of $I$ in $C_1$.
  
  \noindent
  Let $L$ be a line in $C_1$ and $S_L=\Stab{L}{H}$.
  By a change of coordinate we can assume $L$ to be the line at infinity of $\mathbb{P}^2$ and so  $S_L\subset \operatorname{GL}_2$. Another way to rephrase this is to observe that $\Stab{L}{\operatorname{PGL}_3(\field)}\cong \operatorname{GL}_2(\field) \ltimes \field^2$.
  Using $S_L\subset \operatorname{GL}_2(\field)$ and Proposition \ref{pro-GinGL2-trivial}, one gets $\{[\nicefrac{L}{S_L}]\}=\{\nicefrac{L}{S_L}\}$.

  \noindent
  We set $L'=L \cap C_1^*$. 
  Then $[\nicefrac{L}{S_L}]=[\nicefrac{L'}{S_L}]\cup [\nicefrac{L\setminus L'}{S_L}]$.
  By what we said in the zero-dimensional case $\{[\nicefrac{L\setminus L'}{S_L}]\}=\{\nicefrac{L\setminus L'}{S_L}\}$ and so $\{[\nicefrac{L'}{S_L}]\}=\{\nicefrac{L'}{S_L}\}$.
  We call $O_j'$ the orbit of $L_j'$ under $H$. 
  Since $C_1^*$ is the disjoint union of a finite number of orbits $O_j'$, then
  \begin{eqnarray*}
    \{[\nicefrac{C_1}{H}]\}&=&\{[\nicefrac{C_1^*}{H}]\}+\{[\nicefrac{I}{H}]\}=\sum_j\{[\nicefrac{O_j'}{H}]\}+\{[\nicefrac{I}{H}]\}\\
    &=&\sum_j\{[\nicefrac{L_j'}{S_{L_j}}]\}+\{[\nicefrac{I}{H}]\}\\
    &=&\sum_j\{\nicefrac{L_j'}{H}\}+\{\nicefrac{I}{H}\}=\{\nicefrac{C_1}{H}\}.  
  \end{eqnarray*}
  Summarizing the proven facts, one has
  \[
    \{[\nicefrac{\mathbb{P}^2}{H}]\}=\{[\nicefrac{U}{H}]\}+\{[\nicefrac{C_0}{H}]\}+\{[\nicefrac{C_1}{H}]\}=\{\nicefrac{U}{H}\}+\{\nicefrac{C_0}{H}\}+\{\nicefrac{C_1}{H}\}= \{\nicefrac{\mathbb{P}^2}{H}\}.
  \]
  Therefore there remains to prove that $\{\nicefrac{\mathbb{P}^2}{H}\}=\{\mathbb{P}^2\}$.
  For this purpose let $X$ be a resolution of the singularities of $\nicefrac{\mathbb{P}^2}{H}$, $\pi:X\rightarrow \nicefrac{\mathbb{P}^2}{H}$.
   Castelnuovo's theorem implies every unirational surface is rational in characteristic zero (see Chapter V of \cite{HartshorneAlgebraicGeometry})
  and one can construct a birational morphism, $\pi':X\rightarrow \mathbb{P}^2$.
  %
  The quotient singularities of $\nicefrac{\mathbb{P}^2}{H}$ are rational singularities and the exceptional divisor $D_y$ of $y\in \Singu{\nicefrac{\mathbb{P}^2}{H}}$ is a tree of $\mathbb{P}^1$ (see for instance \cite{SpivakovskySandwichedsingularities}).
  This implies that $D_y=\cup_{j=1}^{n_y} \mathbb{P}^1$, where $n_y$ is the number of irreducible components of $D_y$. 
  Then $\{D_y\}=n_y\{\mathbb{P}^1\}-\sum \{*\}$.
  (The sum runs over the intersection points of the copies of $\mathbb{P}^1$.)
  
  Since the graph of the resolution is a tree, then there are exactly $n_y-1$ intersection points in $\sum \{*\}$. Hence $\{D_y\}=n_y\{\mathbb{P}^1\} -(n_y-1)\{*\}=n_y\Lclass+\{*\}$.
  Then,
  \begin{eqnarray*}
    \{\nicefrac{\mathbb{P}^2}{H}\}&=&\{X\}-\sum_{y}\left(\{D_y\} -\{y\}\right)=\{X\}-\sum_{y}\left(n_y\Lclass+\{*\} -\{*\}\right)\\
				  &=&\{X\}-\Lclass\sum_{y}n_y=\{X\}-\Lclass n,
  \end{eqnarray*}
  where $n=\sum_{y}n_y$ is the number of irreducible components in the full exceptional divisor $D=\cup_y D_y$.
  Similarly, one gets $\{\mathbb{P}^2\}=\{X\}-\Lclass m$,
   where $m$ is the number of irreducible components in the full exceptional divisor $E$ of the resolution $X \xrightarrow{\pi'} \mathbb{P}^2$.  
  
  We shall prove that $m=n$.
  Let us consider the following spectral sequence from the map $\pi:X\rightarrow \nicefrac{\mathbb{P}^2}{H}$:
  \[
    E_2^{i,j}=\Hom{i}{\nicefrac{\mathbb{P}^2}{H}}{R^{j}\pi \ldotp \mathbb{Q}_{X}}\Rightarrow \Hom{i+j}{X}{\mathbb{Q}}.
  \]
  Since the map is an isomorphism away from a finite number of points, $R^{j}\pi \ldotp \mathbb{Q}_{X}$ 
  is zero elsewhere but those points.
  If $y$ is one of those points, 
  then $\Hom{i}{\nicefrac{\mathbb{P}^2}{H}}{(R^{j}\pi \ldotp \mathbb{Q}_{X})_y}$ equals $\Hom{i}{\pi^{-1}(y)}{\mathbb{Q}}$ and so $\Hom{0}{\pi^{-1}(y)}{\mathbb{Q}}=\mathbb{Q}^n$ and for $i>0$, $\Hom{i}{\pi^{-1}(y)}{\mathbb{Q}}=0$.
  The spectral sequence degenerates and then we obtain
  $$0 \rightarrow \rat \rightarrow \Hom{2}{X}{\rat} \rightarrow \rat^n \rightarrow 0.$$ 
  This implies $\Hom{2}{X}{\rat}=\rat^{n+1}$.
  Similarly, for $\pi':X\rightarrow \mathbb{P}^2$, one gets $\Hom{2}{X}{\rat}=\rat^{m+1}$ and, thus, the equality $m=n$.
\end{proof}

It is natural to ask if the previous theorem remains valid if $\operatorname{GL}_{3}(\field)$ is replaced by $\operatorname{GL}_{n}(\field)$, for $n \geq 4$. 
For large enough $n$, it will fail.  
For $n = 4$ one can modify the first part of the proof of Theorem 2.4 to show that  $\{[\nicefrac{\mathbb{P}^3}{H}]\}=\{\nicefrac{\mathbb{P}^3}{H}\}\in \grotcom{\Var{\field}}$. 
We do not know whether or not $\{\nicefrac{\mathbb{P}^3}{H}\}=\{\mathbb{P}^3\}$ or whether Theorem \ref{thm-GinGL3-trivial} remains valid for $n = 4$.
Indeed, proving that $\{\nicefrac{\mathbb{P}^3}{H}\}=\{\mathbb{P}^3\}$ involves the study of the resolution of the quotients $\nicefrac{\mathbb{A}^{3}_{\field}}{A}$ for certain $A\subset \operatorname{GL}_{3}(\field)$.



\section{Finite groups with abelian projective quotient}\label{sec-abelian-subgroups}
In this section the ground field is assume to be $\mathbb{C}$.
In what follows we are going to use the following fact. 

\begin{lem}\label{lem-toric-sing-X-G}
  Let $Y$ be a smooth quasi-projective complex algebraic variety and let $A$ be a finite group of automorphisms of $Y$. 
  Let $Y\rightarrow \nicefrac{Y}{A}$ be the canonical quotient map and $\bar{y}$ be the image of $y$.
  
  %
  %
  Let $\bar{y}\in \nicefrac{Y}{A}$. 
  The germ $(\nicefrac{Y}{A}, \bar{y})$ is a simplicial toroidal singularity (i.e. locally isomorphic, in the analytic topology, to the origin in a simplicial toric affine variety) if and only if the quotient $\nicefrac{\Stab{y}{A}}{\pseudo{\Stab{y}{A}}}$ in $T_yY$ is abelian.
\end{lem}
\noindent
A proof can be found, for instance, in the lemma in Section 1.3 of \cite{Pouyanne}.

As in the previous section $G$ is a subgroup of $\operatorname{GL}_n(\com)$ and $H$ is its quotient in $\operatorname{PGL}_n(\com)$.
If $H$ is abelian and if the singularities of $\nicefrac{\mathbb{P}^{n-1}}{H}$ are zero dimensional, then, for the previous lemma, such singularities are toroidal. 
Under this conditions, the Ekedahl invariants satisfy a recursive equation.
\begin{teo}\label{thm-abelian-P-H-X}
  

  If $H$ is abelian and if $\nicefrac{\mathbb{P}^{n-1}_{\mathbb{C}}}{H}$ has only zero dimensional singularities, then for every integer $k$
  \begin{equation}\label{eq-recurrence-ek}
    \eke{k}{G}+\eke{k+2}{G}+\dots+\eke{k+2(n-1)}{G}=\{\Hom{-k}{X}{\inte}\},
  \end{equation}
  where $X\rightarrow\nicefrac{\mathbb{P}^{n-1}}{H}$ is a smooth and proper resolution with normal crossing toric exceptional divisors.

\end{teo}

\noindent
We first show a technical lemma.
We denote by $p_X(t)=\sum_{i\geq 0} \beta^{i}(X)t^i$ the virtual Poincar\'e polynomial of a complex algebraic scheme $X$, where $\beta^{i}(X)=\operatorname{dim}(\Hom{i}{X}{\rat})$ is the $i$-th Betti number of $X$ (here we use the notation of Section 4.5 of \cite{Fu}).
For every smooth projective toric variety $Y$, $p_Y(t)$ is an even polynomial (see Section 5.2 of \cite{Fu}). 

\noindent
We observe that $\Hom{*}{\nicefrac{\mathbb{P}^{n-1}}{H}}{\rat}=\Hom{*}{\mathbb{P}^{n-1}}{\rat}^{H}$ because $H$ is a finite group on the smooth scheme $\mathbb{P}^{n-1}$ and the cardinality of $H$ is invertible in $\rat$. 
We remark that $\Hom{*}{\mathbb{P}^{n-1}}{\rat}=\Hom{*}{\mathbb{P}^{n-1}}{\inte}\otimes \rat$. 
Any element of $\Hom{*}{\mathbb{P}^{n-1}}{\rat}$ can be written as $\sum_{i=1}^{n-1} a_i h^i$ where $h$ is the first Chern class of $O_{\mathbb{P}^{n-1}}(1)$ and $a_i\in \rat$. 
Now, we observe the action is trivial on the coefficients $a_i$ and it comes from a linear action in $\operatorname{GL}_n$ and so $\Hom{*}{\mathbb{P}^{n-1}}{\rat}^{H}=\Hom{*}{\mathbb{P}^{n-1}}{\rat}$.
Hence, if $G$ is a finite subgroup of $\operatorname{GL}_n$, then $p_{\nicefrac{\mathbb{P}^{n-1}}{H}}(t)=p_{\mathbb{P}^{n-1}}(t)$.
%

\begin{lem}\label{lem-technical-lemma-betti-numebers}
  Let $G$, $H$, $\nicefrac{\mathbb{P}^{n-1}}{H}$ and $X$ satisfy the hypothesis of Theorem \ref{thm-abelian-P-H-X}.
  Then:
  \begin{description}
    \item[i)] $\cl(1+\Lclass+\dots+\Lclass^{n-1})=\{\nicefrac{\mathbb{P}^{n-1}}{H}\}$ and, in particular, 
      \begin{equation}\label{eq-recurrence-ek-quasi}
	\eke{k}{G}+\eke{k+2}{G}+\dots+\eke{k+2(n-1)}{G}=\mathcal{H}^{-k}\left(\{\nicefrac{\mathbb{P}^{n-1}}{H}\}\right).
      \end{equation}
      
    \item[ii)] 
      Every singularity of $\nicefrac{\mathbb{P}^{n-1}}{H}$ is a toroidal singularity and 
      \begin{equation}\label{eq-P-n-1-H-singolarita-isolate-abeliano}
	\{\nicefrac{\mathbb{P}^{n-1}}{H}\}=\{X\}-\sum_y\left(\{D_y\}-\{y\}\right),
      \end{equation}
      where the sum runs over $y\in \Singu{\nicefrac{\mathbb{P}^{n-1}}{H}}$;
      $D_y=\pi^{-1}(y)$ is the exceptional toric divisor of the resolution of $y$ with irreducible components decomposition $D_y=D^1_y\cup \dots \cup D^r_y$; $\{D_y\}=\sum_{q\geq 1}(-1)^{q+1}\sum_{i_1<\dots<i_q} \{D^{i_1}_y\cap \dots \cap D^{i_q}_y\}$.
      
    \item[iii)]
      If $k$ is non-zero and even, one has
      \[
	1=\beta^{k}(X)-\sum_y\sum_{q\geq 1}(-1)^{q+1}\sum_{i_1<\dots<i_q} \beta^{k}(D^{i_1}_y\cap \dots \cap D^{i_q}_y)
      \]
      and, for $k=0$, 
      \[
	1=\beta^{0}(X)-\sum_y\sum_{q\geq 1}(-1)^{q+1}\sum_{i_1<\dots<i_q} \left(\beta^{0}(D^{i_1}_y\cap \dots \cap D^{i_q}_y)-1\right).
      \]
      
    \item[iv)] 
      $\beta^{odd}(X)=0$.
  \end{description}
%
%
%
\end{lem}
\begin{proof}
  As in the proof of Proposition \ref{pro-GinGL2-trivial} and of Theorem \ref{thm-GinGL3-trivial}, we consider all the subvarieties of $\mathbb{P}^{n-1}$ globally fixed by certain subgroup $A$ of $H$. This subgroup is abelian and by Proposition 3.2 in \cite{Ekedahl-inv} $\{\B{A}\}=\{*\}$.
  Thus, $\{[\nicefrac{\mathbb{P}^{n-1}}{H}]\}=\{\nicefrac{\mathbb{P}^{n-1}}{H}\}$.
  Using in order, $\cl\left(1+\Lclass+\dots+\Lclass^{n}\right)=\{[\nicefrac{\mathbb{P}(V)}{G}]\}$ in $\grotcom{\Var{\field}}$, Proposition \ref{prop-equality-classes-PV} and the latter equality, we obtain the first part of \textbf{i)}.

  %
  %
  %
  For the second part of \textbf{i)}, we note that applying the cohomological map $\mathcal{H}^{-k}$ on the left hand side, one has:
  \begin{eqnarray*}
     \mathcal{H}^{-k}\left(\{\B{G}\}(1+\dots+\Lclass^{n-1})\right)
	&=&\mathcal{H}^{-k}\left(\{\B{G}\}\right)+\dots+\mathcal{H}^{-k}\left(\{\B{G}\}\Lclass^{n-1}\right)\\
	&=&\mathcal{H}^{-k}\left(\{\B{G}\}\right)+\dots+\mathcal{H}^{-k-2(n-1)}\left(\{\B{G}\}\right)\\
	&=&\eke{k}{G}+\dots +\eke{k+2(n-1)}{G}.
   \end{eqnarray*}
  
  Regarding item \textbf{ii)}: Every stabilizer group of $H$ is abelian and so it is for the quotient of $\Stab{x}{H}$ modulo $\pseudo{\Stab{x}{H}}$ in $T_xX$.
  Then, by Lemma \ref{lem-toric-sing-X-G}, each singularity of $\{\nicefrac{\mathbb{P}^{n-1}}{H}\}$ is an isolated simplicial toroidal singularities. 
  One produces a toric resolution $X$ with normal crossing toric exceptional divisors (see Section 2.6 of \cite{Fu}). 
  This means that each intersection $D^{i_1}_y\cap \dots \cap D^{i_q}_y$ is a smooth toric variety.
  The resolution $X\rightarrow  \nicefrac{\mathbb{P}^{n-1}}{H}$ restricted to the pull back of $\nicefrac{\mathbb{P}^{n-1}}{H}\setminus \Singu{\nicefrac{\mathbb{P}^{n-1}}{H}}$ is an isomorphism and then $\{X\}-\sum_y \{D_y\}=\{\nicefrac{\mathbb{P}^{n-1}}{H}\}-\sum_y \{y\}$. 
  
  
  Therefore, $p_{D_y}(t)=\sum_{q\geq 1}(-1)^{q+1}\sum_{i_1,\dots,i_q}p_{D^{i_1}_y\cap \dots \cap D^{i_q}_y}(t)$ and the odd degree coefficients of $p_{D_y}(t)$ are zero.

  Finally, we want to compute the virtual Poincar\'e polynomial of $X$. 
  Via formula (\ref{eq-P-n-1-H-singolarita-isolate-abeliano})
  and using $p_{\nicefrac{\mathbb{P}^{n-1}}{H}}(t)=p_{\mathbb{P}^{n-1}}(t)$, 
  \[
    p_{\mathbb{P}^{n-1}}(t)=p_X(t)-\sum_y(p_{D_y}(t)-1).
  \]
  Comparing, degree by degree, the polynomial in the left hand side and in the right hand side, one gets the Betti numbers equalities and item \textbf{iv)}. 
\end{proof}

\begin{proof}[Proof of Theorem \ref{thm-abelian-P-H-X}]
  From the first item of the previous lemma, we know that (\ref{eq-recurrence-ek-quasi}) holds.
  %
  We shall show that $\mathcal{H}^{-k}\left(\{\nicefrac{\mathbb{P}^{n-1}}{H}\}\right)=\{\Hom{-k}{X}{\inte}\}$.
  Using the previous technical lemma we express $\{\nicefrac{\mathbb{P}^{n-1}}{H}\}$ in (\ref{eq-P-n-1-H-singolarita-isolate-abeliano}) as a sum of smooth and proper varieties 
  and smooth toric varieties $\{D^{i_1}_y\cap \dots \cap D^{i_q}_y\}$.
  
  If $k>0$ or $k<-2(n-2)$, $\mathcal{H}^{-k}\left(\{D_y\}-\{y\}\right)=0$  for dimensional reason and so the recurrence holds.
  The same holds, if $k$ is odd integer between $0$ and $-2(n-2)$, because the cohomology of a smooth toric variety is torsion free by Section 5.2 in \cite{Fu}.
  
  \noindent
  It remains to consider the case $0\leq k=2j \leq -2(n-2)$.
  For these values, in the left hand side of (\ref{eq-recurrence-ek-quasi}), there are certain Ekedahl invariants with negative indices (so zero), $\eke{0}{G}$ and some positive even Ekedahl invariants $\eke{2}{G}+\dots+\eke{2j+2(n-1)}{G}$ that are an integer linear combination of classes of finite abelian groups (we use the second part of Theorem \ref{thm-ekedahl}).
  
  \noindent
  On the right hand side of (\ref{eq-recurrence-ek-quasi}) the only possible torsion part is $\{\operatorname{tor}\Hom{-k}{X}{\inte}\}$, because the cohomologies of a smooth toric variety is torsion free.
  Hence, what remains to prove is that the free parts cancel each others: this is equivalent to item \textbf{iii}) of the previous lemma.
%
%
\end{proof}

\section{The discrete Heisenberg group $H_p$}\label{sec:h-5}
Let $p$ be an odd.
The $p$-discrete Heisenberg group $\hei{p}$ is the following subgroup of $\Gl{3}{\pfield}$:
\[
	\hei{p}=\left\lbrace M(a,b,c)\stackrel{notation}{=}\left(\begin{array}{ccc}
						1 & a & b\\
						0 & 1 & c\\
						0 & 0 & 1
						\end{array}\right):\, a,b,c \in \pfield \right\rbrace.
\]
We observe that $\hei{p}$ is generated by $\mathcal{X}=M(1,0,0)$, $\mathcal{Y}=M(0,0,1)$ and $\mathcal{Z}=M(0,1,0)$ modulo the relations  $\mathcal{Z}\mathcal{Y}\mathcal{X}=\mathcal{X}\mathcal{Y}$, $\mathcal{Z}^p=\mathcal{X}^p=\mathcal{Y}^p=M(0,0,0)$, $\mathcal{Z}\mathcal{X}=\mathcal{X}\mathcal{Z}$ and $\mathcal{Z}\mathcal{Y}=\mathcal{Y}\mathcal{Z}$.
The center of $\hei{p}$, $\centro$, is cyclic and generated by $\mathcal{Z}$. 
We denote by $A_p$ the group quotient $\nicefrac{\hei{p}}{\centro}\cong \Zp\times \Zp$. 
Moreover, $\hei{p}$ is the central extension of $\Zp$ by $\Zp\times \Zp$:
\begin{equation}\label{seq-esatta-Hp}
  1\rightarrow \Zp \rightarrow \hei{p} \stackrel{\phi}{\rightarrow} \Zp\times \Zp\rightarrow 1.
\end{equation}
Using Lemma 4.9 in \cite{Bogomolov1988}, one proves that the Bogomolov multiplier $B_0(\hei{p})$ is zero for every prime $p$.

%

\noindent
We also remark that the discrete Heisenberg group has $p^2+p-1$ irreducible complex representations: $p^2$ of them are one dimensional and the remaining $p-1$ are faithful and $p$-dimensional.

Let $V$ be a faithful irreducible $p$-dimensional complex representation of $\hei{p}$, $\hei{p}\stackrel{\representationname}{\rightarrow} \operatorname{GL}(V)$. 
There is a natural action of $H_p$ on $V$ and it induces an action on the complex $(p-1)$-dimensional space $\mathbb{P}^{p-1}$. 
One so defines the quotient $\nicefrac{\mathbb{P}^{p-1}}{\hei{p}}$.

\noindent
Since $\mathcal{Z}$ belongs to the center, $\representationname\left(\mathcal{Z}\right)=e^{\frac{2\pi i}{p}} \operatorname{Id}$, for some $0< i< p$, where $\operatorname{Id}$ is the identity element of $\operatorname{GL}(V)$.
Hence, the center acts trivially on $\mathbb{P}^{p-1}$ and $\nicefrac{\mathbb{P}^{p-1}}{\hei{p}}\cong\nicefrac{\mathbb{P}^{p-1}}{A_p}$. 
From Lemma \ref{lem-toric-sing-X-G}, we know that if $\nicefrac{\mathbb{P}^{p-1}}{A_p}$ has singularities, then they are toroidal.
We study these singularities and so we focus on $\Stab{x}{A_p}$. 

\begin{pro}\label{pro-nontrivial-stabilizer}
  Let $x\in \mathbb{P}^{p-1}$. 
  If the action of $A_p$ at $x$ is not free, then $|\Stab{x}{A_p}|=p$.
\end{pro}
\begin{proof}
  Call $W_x$ the one dimensional subvector-space of $V$ corresponding to $x$. 
  By assumption, the stabilizer of $x$ is a nontrivial subgroup of $A_p$ and, by Lagrange's Theorem, it could have order $p$ or $p^2$. 
  If $\Stab{x}{A_p}=A_p$, then for every $g \in A_p, \, gW_{x}=W_{x}$ and $A_pW_{x}=W_{x}$. Then $\hei{p}W_{x}=W_{x}$.
  Therefore, $W_{x}$ is an one dimensional irreducible $\hei{p}$-subrepresentation of $V$ contradicting the fact that $\hei{p}$ acts irreducibly.
\end{proof}
There are exactly $p+1$ subgroups of order $p$ in $A_p$. 
Let $B$ be one of them.
We define $\barB$ to be a subgroup of $\hei{p}$ such that the surjection in (\ref{seq-esatta-Hp}) restricted to it, $\phi_{|\barB}$, is a group isomorphism:   $\barB\cong \Zp \subset \phi^{-1}(B)$. 


\noindent
We restrict the representation $\hei{p}\repr \Gl{p}{\mathbb{C}}$ to the subgroup $\barB=\Zp$ and we write $V$ as a direct sum of one dimensional irreducible representations: $V=\oplus_{\chi\in \Zp} V_{\chi}$, where $V_{\chi}=\{v\in V: g \cdot v = \chi(g) \cdot v,\, \forall g\in \Zp\}$. 
In other words, $\barB$ fixes $p$ one dimensional linear subspaces $V_{\chi}$ and so $B$ fixes $p$ points $P_{\chi}\in\mathbb{P}^{p-1}$, with $\Stab{P_{\chi}}{A_p}=B$, that is $(\mathbb{P}^{p-1})^{B}=\{P_{\chi_0}, \dots, P_{\chi_{p-1}}\}$.
\begin{pro}\label{pro-distinct-point}
  If $B$ and $B'$ are two distinct $p$-subgroups of $A_p$, then $(\mathbb{P}^{p-1})^{B}\cap (\mathbb{P}^{p-1})^{B'}=\emptyset$.
\end{pro}
\begin{proof}
  Trivially, under the hypothesis, $A_p=B\oplus B'$ and if $P\in (\mathbb{P}^{p-1})^{B}\cap (\mathbb{P}^{p-1})^{B'}$, then $\Stab{P}{A_p}=A_p$, contradicting Proposition \ref{pro-nontrivial-stabilizer}.
\end{proof}

\noindent
%
We observe that $\nicefrac{A_p}{B}$ acts regularly
on $(\mathbb{P}^{p-1})^{B}$. 
Thus, these points lie in the same orbit under the action of $\nicefrac{A_p}{B}$ and this means that they correspond to a unique point $y_B$ in $\var{p}$.
\begin{teo}\label{theo-singular-point-var-p}
  The quotient $\var{p}$ has $p+1$ simplicial toroidal singular points.
\end{teo}
\begin{proof}
  There are exactly $p+1$ subgroups, $B$, of order $p$ in $A_p$. 
  Each of them corresponds to a point $y_B$ in $\var{p}$.
  By Proposition \ref{pro-distinct-point}, these points are distinct.
  
  \noindent
  Let $y\in \mathbb{P}^{p-1}$ such that $\bar{y}=y_B$.
  We consider the action of $\Stab{y}{A_p}$ on the tangent space $T_{y}\mathbb{P}^{p-1}$. 
  The pseudo-reflection group $\pseudo{\Stab{y}{A_p}}=\{e\}$, because it is a subgroup of $\Stab{y}{A_p}\cong \Zp$ and, so, it is either the trivial group or $\Stab{y}{A_p}$. 
  The latter is not possible because $\Stab{y}{A_p}$ stabilizes only the origin of the vector space $T_{y}\mathbb{P}^{p-1}$.
  Thus, $\pseudo{\Stab{y}{A_p}}\neq \Stab{y}{A_p}$ in $T_{y}\mathbb{P}^{p-1}$ and by Lemma \ref{lem-toric-sing-X-G} these singularities are also toroidal and simplicial.
\end{proof}

We now outline a method to calculate the Ekedahl invariants for $\hei{p}$:
we write $\{\var{p}\}$ as a sum of classes of smooth and proper varieties and we use Theorem \ref{thm-abelian-P-H-X}.

\noindent
Let $\varns{p} \res \var{p}$ be the resolution of the $p+1$ toroidal singularities of $\var{p}$.
One has the following geometrical picture:
\begin{diagram}
		& 					& \mathbb{P}(V)		&\lInto	& U\\
		& 					& \dTo_{\quotientname}	&	&\dTo(0,4)^{\quotientname_{|_U}}\\
\varns{p} 	& \rTo^{\resolutionname} 		& \var{p}		&	& \\
\uInto		&  					& 			&\luInto & \\
\open{p}	& \rTo(4,0)_{\resolutionname_{|_{U_p}}}^{\sim} 	& 			&	& \open{p}
\end{diagram}
where $U$ is the open subset of $\mathbb{P}(V)$ where $A_p$ acts freely; 
$\open{p}=\nicefrac{U}{A_p}$.

Since $A_p$ is abelian, using Theorem \ref{thm-abelian-P-H-X} one gets
\[
  \eke{k}{G}+\eke{k+2}{G}+\dots+\eke{k+2(p-1)}{G}=\{\Hom{-k}{\varns{p}}{\inte}\}.
\]
Because of Theorem \ref{thm-ekedahl}, $\eke{0}{\hei{p}}=\{\inte\}$ and $\eke{1}{\hei{p}}=\eke{2}{\hei{p}}=0$.
Thus, we focus on $\eke{3}{\hei{p}}$.
We are going to show that $\eke{3}{\hei{5}}=\eke{4}{\hei{5}}=0$.
We set $p=5$ because of the difficulty in computing $\operatorname{tor}(\Hom{*}{\varns{p}}{\inte})$ for $p>5$.
\begin{claim} 
$\operatorname{tor}(\Hom{5}{\varns{5}}{\inte}) = 0$.
\end{claim}

\noindent
Using this claim, we prove the main result.
\begin{teo}\label{thm-ekedahl-invariants-H-p}
  The Ekedahl invariants of the fifth discrete Heisenberg group $\hei{5}$ are trivial.
\end{teo}
\begin{proof}
  By using Theorem \ref{thm-abelian-P-H-X} for $G=\hei{5}$, $n=5$, $k=-2\cdot 5+5$ and $X=\varns{5}$ and also by applying the second part of Theorem \ref{thm-ekedahl}, we have $\eke{3}{\hei{5}}
  =\{\operatorname{tor}\left(\Hom{5}{\varns{5}}{\inte}\right)\}$.
  By Poincar\'e duality,
  $$\eke{3}{\hei{5}}=\{\operatorname{tor}\left(\Hom{5}{\varns{5}}{\inte}\right)\}
  =\{\operatorname{tor}\left(\Hom{4}{\varns{5}}{\inte}\right)\}$$ 
  and this is zero by the claim. 
  Similarly, for $k=-2\cdot 5+6$,
  \[
      \eke{4}{\hei{5}}=\{\operatorname{tor}\left(\Hom{2\cdot 5-6}{\varns{5}}{\inte}\right)\}
  =\{\operatorname{tor}\left(\Hom{4}{\varns{5}}{\inte}\right)\}.
  \]
  
  \noindent
  Moreover, $\eke{i}{\hei{5}}=0$ for $i\geq 5$.
  Indeed, $\eke{5}{\hei{5}}=\{\operatorname{tor}\left(\Hom{3}{\varns{5}}{\inte}\right)\}=\eke{2}{\hei{5}}=0$ and $\eke{6}{\hei{5}}=\{\operatorname{tor}\left(\Hom{2}{\varns{5}}{\inte}\right)\}=\eke{1}{\hei{5}}=0$.
  In addition, $\eke{i}{\hei{5}}=0$ for $i>6$ for dimensional reason.
%
\end{proof}

\noindent
We observe that the same proof would follow for $\hei{p}$ for $p>5$ if we had enough information about the vanishing of the torsion in the cohomology of $X_p$.
This fact and the recent proof (by Kang in \cite{Kang-Rationality}) of the rationality of the extension $\com(x_g, g\in \hei{p})^{\hei{p}}/\com$ suggest the following conjecture:

\begin{conj}
  The Ekedahl invariants of the Heisenberg group $H_p$ of order $p^3$ are trivial.
\end{conj}

\subsection{Proof of the claim}
To obtain the claim it sufficies to show that $\Hom{odd}{\open{5}}{\inte}=0$ and $\Hrel{5}{\varns{5}}{E}{\inte}=0$, where $E$ is the union of exceptional divisors of the resolution $\varns{5} \res \var{5}$.
\begin{teo}\label{teo-cohomology-open-set-articolo}
  The cohomology of the smooth open subset $\open{p}$ of $\nicefrac{\mathbb{P}(V)}{A_p}$ for $k<2p-2$ is 
  \[
    \Hom{k}{\open{p}}{\inte}=\begin{cases}
				\inte 			&\text{if $k=0$};\\
				0 			&\text{if $k$ is odd};\\
				\inte\oplus (\Zp)^{\frac{k}{2}+1}	&\text{if $k\neq 0$ and even}.
	                      \end{cases}                      
  \]
\end{teo}
\begin{proof}
  Since $A_p$ acts freely on $U$, let us consider the Cartan-Leray spectral sequence (see Section 5 or Theorem $8^{bis}.9$ in \cite{McCleary-guide}) relative to the quotient map $\quotientname: U\rightarrow \open{p}$:
  \[
    E_2^{i,j}=\Hom{i}{A_p}{\Hom{j}{U}{\inte}}\Rightarrow \Hom{i+j}{\open{p}}{\inte}.
  \]
  Let $\Sp=\mathbb{P}^{p-1}\setminus U$. 
  One sees that $\Hom{i}{U}{\inte}\cong \Hom{i}{\mathbb{P}(V)}{\inte}$ for $i<2p-3$, $\Hom{2p-3}{U}{\inte}=\inte[\Sp]^0$ and $\Hom{2p-3}{U}{\inte}$ is zero otherwise.
  Here $\inte[\Sp]$ is the group freely generated by the $p(p+1)$ points in $\Sp$ and $\inte[\Sp]^0$ is the kernel of the argumentation map.
  
  %
  %
  \noindent
  To read the $E_2^{i,j}$-terms we observe that the cohomology of $A_p$ has a $\inte$-algebra structure:
  \[
    \Hom{*}{A_p}{\inte}\cong \frac{\inte[x_1,x_2,y]}{(y^2, px_1, px_2, py)},
  \]
  where $\operatorname{deg}(x_1)=\operatorname{deg}(x_2)=2$ and $\operatorname{deg}(y)=3$. 
  Indeed the $\inte$-algebra structure comes from the Bockstein operator for $\Hom{*}{\Zp}{\pfield}$. (The reader may find a detailed proof in Appendix A of \cite{Thesis-Martino}.)
  
  In this proof, we only care about the terms $E_2^{i,j}$ with $j<2p-3$ where the differential $d_2^{i,j}$ is zero.
  
  \noindent
  Let $h$ be the first Chern class of $O_{\mathbb{P}^{p-1}}(1)$ (hence $h^k$ generates $\Homred{2k}{\mathbb{P}^{p-1}}$). 
  Using this notation one writes an element of $E_2^{i,2k}$ (for $2k<2p-3$) as $h^k\cdot \alpha$ where $\alpha \in \Hom{i}{A_p}{\Hom{2k}{U}{\inte}}$.
  This keeps track of the differential $\Hom{*}{\mathbb{P}^{p-1}}{\inte}$-algebra structure (see Chapter 2 in \cite{McCleary-guide}): for $2k<2p-3$, $d_3(h^{k})=k \cdot h^{k-1}d_{3}(h)$. 
  In addition, $d_3(h)$ is a non zero multiple of $y$ in $\Hom{3}{A_p}{\inte}$. 
  Thus, the differential $d_3$ is a degree $3$ homomorphism from $\Hom{*}{A_p}{\inte}$ to itself:
  \begin{eqnarray*}
    d_3^{i,2k}:\Hom{i}{A_p}{\inte}	&\rightarrow & \Hom{i+3}{A_p}{\inte}\\
				  1&\mapsto     & \alpha y.
  \end{eqnarray*}
  for $i,k>0$.
%
%
  Then, one computes that for $0\leq j <2p-4$, one has  
  \[
    E_4^{i,j}\cong\begin{cases}
		    \Hom{i}{A_p}{\inte}	& \text{ for $i$ even and $j=0$;}\\
		    \inte		& \text{ for $j\leq 2p-4$ even and $i=0$;}\\
		    0			& \text{ otherwise}
		  \end{cases}
  \]
  and $E_4^{0,2p-4}=\inte$ and $E_4^{1,2p-4}=0$.
  
  \noindent
  Moreover, the spectral sequence degenerates, $E_{\infty}^{i,j}=E_{4}^{i,j}$, for $j=0$ with $i<2p-2$, $0<j<2p-4$ and $j=2p-4$ with $i=0,1$.
  We only remark that from the $E_{\infty}$-level, one reads the information about $\Hom{*}{\open{p}}{\inte}$ via $E_{\infty}^{i,j}=\operatorname{gr}\left(\Hom{i+j}{\open{p}}{\inte}\right)$. 
\end{proof}
The following result is also true for every prime $p$.
\begin{teo}\label{theo-toroidal-singularities-articolo}
  Each singularity of $\var{p}$ is locally isomorphic to the origin of the toric affine variety $\nicefrac{\mathbb{A}^{p-1}}{\Zp}$.
  Moreover it is of the type $\frac{1}{p}(1,2,\dots, p-1)$, that is $\Zp$ acts on $\mathbb{A}^{p-1}$ diagonally via $j\mapsto \operatorname{diag}(\zeta^{j}, \zeta^{2j}, \dots, \zeta^{(p-1)j})$ where $\zeta$ is a primitive $p$-root of unity.
\end{teo}
\begin{proof}
  In Theorem \ref{theo-singular-point-var-p} we have already shown that $\var{p}$ has $p+1$ zero dimensional simplicial toroidal singularities.
  We keep using the notation in the beginning of Section \ref{sec:GL23}.
  
  Let $V$ decompose as a direct sum of one dimensional irreducible representations, $V=\oplus_{\chi\in \Zp} V_{\chi}$. 
  %
  A point $P_{\chi'}$ in $(\mathbb{P}^{p-1})^B$ corresponds to some $V_{\chi'}$ for some character $\chi'$ and so $\nicefrac{V}{V_{\chi'}}=\oplus_{\chi\in \barB^{\vee}, \chi \neq \chi'} V_{\chi}$.
  
  \noindent
  The action of $\nicefrac{A_p}{B}$ on $(\mathbb{P}^{p-1})^B$ is regular and using Cartan's Lemma, the germ $(\var{p}, y_B)$ is locally isomorphic to $(\nicefrac{T_{P_{\chi'}}\mathbb{P}^{p-1}}{\Stab{P_{\chi}}{A_p}}, \bar{O})$,  where $\bar{O}$ is the image of the origin of $T_{P_{\chi'}}\mathbb{P}^{p-1}$ under the quotient map $$T_{P_{\chi'}}\mathbb{P}^{p-1}\rightarrow \nicefrac{T_{P_{\chi'}}\mathbb{P}^{p-1}}{\Stab{P_{\chi'}}{A_p}}.$$
  
  \noindent
  Without loss of generality let $\chi'=1$.
  Using those facts we obtain $T_{P_{\chi'}}\mathbb{P}^{p-1}=\bigoplus_{1\neq\chi \in  \Zp} \com e_{\chi}$.
  Thus, the action of $\Zp$ on the tangent space is given by $g\cdot v= \chi(g)v$ for any $v\in V_{\chi}$.
  Therefore $\Zp$ acts via the homomorphism $\Zp  \hookrightarrow  (\mathbb{C}^{*})^{p-1}$ sending $j$ to $(\zeta^{j}, \zeta^{2j}, \dots, \zeta^{(p-1)j})$,
  where $\zeta$ is a $p$-root of unity.
  %
\end{proof}

To proceed we need a resolution for such toroidal singularities and for this reason we have to set $p=5$. 

\noindent
We construct the toric resolution of $\nicefrac{\mathbb{A}^{4}}{\nicefrac{\inte}{5\inte}}$ using the computer algebra program \texttt{Magma}.
%
The resolution is made by adding a suitable number of rays to the single cone of the fan of $\nicefrac{\mathbb{A}^{4}}{\nicefrac{\inte}{5\inte}}$ (see Exercise on page 35 of \cite{Fu}).
The new fan is denoted by $\Delta_5$ and consists of \textbf{$10$} rays and \textbf{$21$} maximal cones. 
To each ray one associated a toric divisor, $D_i=V(\operatorname{Star}(r_i))$, in the resolution of $\nicefrac{\mathbb{A}^{4}}{\nicefrac{\inte}{5\inte}}$. 
There are $6$ new rays that correspond to $6$ exceptional divisors $D_i$ for $1\leq i\leq 6$. 
Let $D$ be the union of them.
One has that if $k>3$, then $D_{i_1}\cap \dots \cap D_{i_k}=\emptyset$ (that is there exist no maximal cones of $\Delta_5$, generated only by exceptional divisors).
We will not write down the rays and the cones in $\Delta_5$; the details of this computation can be found in Chapter 6 and 7 of \cite{Thesis-Martino}. 
An example of fan of $D_1$ is in Figure \ref{fig:orderingD5}.

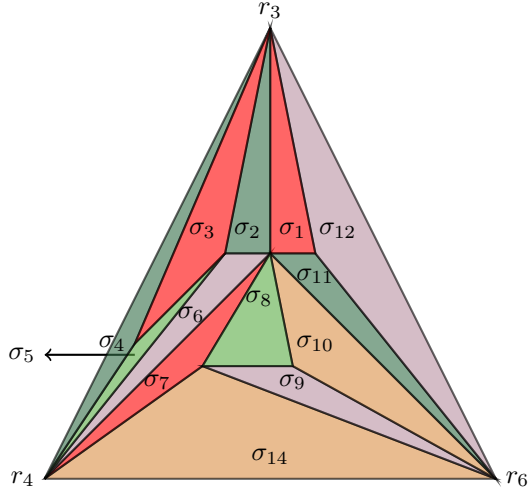
\begin{figure}[htb]
\centering
\begin{tikzpicture}[thick,scale=3]
\coordinate (4) at (0,0);
\coordinate (6) at (2,0);
\coordinate (3) at (1,2);
\coordinate (2) at (1,1);

\coordinate (1) at (0.7,0.5);
\coordinate (8) at (1.1,0.5);

\coordinate (9) at (0.8,1);
\coordinate (10) at (0.4,0.6);

\coordinate (7) at (1.2,1);

\draw[->] (0.4,0.55)--(0,0.55);

\node[left] at (4) {$r_4$};
\node[right] at (6) {$r_6$};
\node[above] at (3) {$r_3$};

\draw[fill=cof,opacity=0.6] (4) -- (1) -- (6)--cycle;
\draw[fill=red,opacity=0.6] (4) -- (1) -- (2)--cycle;
\draw[fill=pur,opacity=0.6] (4) -- (2) -- (9)--cycle;
\draw[fill=greeo,opacity=0.6] (4) -- (9) -- (10)--cycle;
\draw[fill=greet,opacity=0.6] (4) -- (10) -- (3)--cycle;

\draw[fill=pur,opacity=0.6] (6) -- (1) -- (8)--cycle;
\draw[fill=cof,opacity=0.6] (6) -- (8) -- (2)--cycle;
\draw[fill=greet,opacity=0.6] (6) -- (2) -- (7)--cycle;
\draw[fill=pur,opacity=0.6] (6) -- (7) -- (3)--cycle;

\draw[fill=red,opacity=0.6] (3) -- (7) -- (2)--cycle;
\draw[fill=greet,opacity=0.6] (3) -- (2) -- (9)--cycle;
\draw[fill=red,opacity=0.6] (3) -- (9) -- (10)--cycle;

\draw[fill=greeo,opacity=0.6] (1) -- (8) -- (2)--cycle;

\node at (1.1,1.1) {$\sigma_1$};
\node at (0.9,1.1) {$\sigma_2$};
\node at (0.7,1.1) {$\sigma_3$};
\node at (0.3,0.6) {$\sigma_4$};
\node at (-0.1,0.55) {$\sigma_5$};
\node at (0.65,0.73) {$\sigma_6$};
\node at (0.5,0.43) {$\sigma_7$};
\node at (0.95,0.8) {$\sigma_8$};
\node at (1.1,0.43) {$\sigma_9$};
\node at (1.2,0.6) {$\sigma_{10}$};
\node at (1.2,0.9) {$\sigma_{11}$};
\node at (1.3,1.1) {$\sigma_{12}$};
\node at (1,0.1) {$\sigma_{14}$};

\end{tikzpicture}
  \caption{The fan of the three dimensional toric variety $D_1$. 
  The fan is complete and we show the associated triangularization of the $2$-dimensional sphere (the simplex $r_4, r_3, r_6$ closes the sphere). 
  We show a shelling order of this triangulation that is an ordering of the maximal cone in the fan of $D_1$ satisfying property $(*)$ on page 101 of \cite{Fu}. 
  The $2$-simplex having vertices $r_4$, $r_3$ and $r_6$ is $\sigma_{13}$.}\label{fig:orderingD5}
\end{figure}

\begin{teo}\label{thm-rational-cohomology-exceptional-divisor-articolo}
  $\Hrel{5}{\varns{5}}{E}{\inte}=0$.
\end{teo}
\begin{proof}
  Let $E=\resolutionname^{-1}(\Singu{\nicefrac{\mathbb{P}^{4}}{A_5}})$ be the union of exceptional divisors from the resolution of the six toroidal singularities locally isomorphic to $\nicefrac{\mathbb{A}^{4}}{\nicefrac{\inte}{5\inte}}$.
  Then $E=\sqcup_{y\in \Singu{\nicefrac{\mathbb{P}^{4}}{A_5}}}E^{(y)}$, with $E^{(y)}=\resolutionname^{-1}(y)$, and each $E^{(y)}$ is isomorphic to $D$.
  Thus $$\Hrel{*}{\varns{5}}{E}{\inte}=\bigoplus_{y \in \Singu{\nicefrac{\mathbb{P}^{4}}{A_5}}} \Hrel{*}{\varns{5}}{D}{\inte}=\Hrel{*}{\varns{5}}{D}{\inte}^{\oplus 6}.$$
  %
  Since we want to prove that $\Hrel{5}{\varns{5}}{E}{\inte}=0$, from now on we focus on $\Hrel{*}{\varns{5}}{D}{\inte}$. 

  We denote $D_{i_1}\cap D_{i_2}\cap \dots \cap D_{i_k}$ by $D_{i_1, i_2, \dots, i_k}$.
  To compute $\Hrel{*}{\varns{5}}{D}{\inte}$, we use the second quadrant spectral sequence 
  \[
    E_1^{-k,i}=\bigoplus_{i_1<\dots<i_k} \Hrel{i}{\varns{5}}{D_{i_1, i_2, \dots, i_k}}{\inte}\Rightarrow \Hrel{*}{\varns{5}}{D}{\inte}.
  \]
  The $E_1$-terms are defined for every $i$ and for $k>0$.  

  We first observe that $D_{i_1, i_2, \dots, i_k}$ is a smooth toric variety corresponding to the \emph{star} of the cone $\langle r_{i_1}, \dots, r_{i_k}\rangle$, so $$\Hrel{*}{\varns{5}}{D_{i_1, \dots, i_k}}{\inte}=\Hom{*-2\operatorname{dim}(\langle r_{i_1}, \dots, r_{i_k}\rangle)}{D_{i_1, \dots, i_k}}{\inte}.$$
%
  Recalling that there are at most triple intersections, we immediately have that $E_1^{-k,i}=0$ if $k>3$.

  \noindent
  Table \ref{tab:differential-d-1-articolo} shows the $E_1$-terms of the spectral sequence and the non-zero differen\-tials.
  \begin{table}[ht]
  \begin{center}
  \begin{tabular}{ccccc||c}
  0		&					&0 		&					&0		&9\\  
  $\oplus \Homred{2}{D_{i_1}\cap D_{i_2} \cap D_{i_3}}$	&$\stackrel{d_1^{3,8}}{\rightarrow}$	& $\oplus \Homred{4}{D_{i_1}\cap D_{i_2}}$ 	&$\stackrel{d_1^{2,8}}{\rightarrow}$	& $\oplus \Homred{6}{D_{i_1}}$ 	&8\\  
  0 		&					&0 		&					&0		&7\\  
  $\oplus \Homred{0}{D_{i_1}\cap D_{i_2} \cap D_{i_3}}$ 	&$\stackrel{d_1^{3,6}}{\rightarrow}$	& $\oplus \Homred{2}{D_{i_1}\cap D_{i_2}}$ 	&$\stackrel{\mathbf{d_1^{2,6}}}{\rightarrow}$	& $\oplus \Homred{4}{D_{i_1}}$ 	&6\\  
  0 		&					&0 		&					&0		&5\\  
  0 		&					& $\oplus \Homred{0}{D_{i_1}\cap D_{i_2}}$ 	&$\stackrel{d_1^{2,6}}{\rightarrow}$	&$\oplus \Homred{2}{D_{i_1}}$ 	&4\\  
  0 		&					&0 		&					&0		&3\\  
  0 		&					&0 		&					&$\oplus \Homred{0}{D_{i_1}}$	&2\\  
  0 		&					&0 		&					&0		&1\\  
  0 		&					&0 		&					&0 		&0\\ \hline\hline 
  -3&					&-2		&					&-1		&$-k\diagdown i$
  \end{tabular}
  \end{center}
  \caption{The $E_1$-terms and the differentials $d_1$.}\label{tab:differential-d-1-articolo}
  \end{table}
  All the indexes of the direct sums run over the indicies of the exceptional divisors $D_i$. 
  All the cohomologies are integral cohomologies.
  %
  %
  %

  Now, we focus on the homomorphism 
  \[
    \oplus \Homred{2}{D_{i_1, i_2}}\stackrel{\mathbf{d_1^{2,6}}}{\rightarrow}\oplus \Homred{4}{D_{i_1}}.
  \]
  Let us assume that $\operatorname{Ker}(d_1^{2,6})=0$.  
  Then $E_2^{2,6}=0$ and so $E_{\infty}^{2,6}=E_2^{2,6}=0$. 
  We remark that the spectral sequence converges to $\Hrel{k+i+1}{\varns{5}}{D}{\inte}$, that is $$E_{\infty}^{k,i}=\operatorname{gr}\left(\Hrel{k+i+1}{\varns{5}}{D}{\inte} \right).$$ 
  (Since columns are counted from $k=-1$, there is a shift by one.)
  The terms of the spectral sequence involved in $\operatorname{gr}\left(\Hrel{5}{\varns{5}}{D}{\inte} \right)$ are $E_{\infty}^{-1,5}$, $E_{\infty}^{-2,6}$ and $E_{\infty}^{-3,7}$ and all of them are zero.
  Thus, $\Hrel{5}{\varns{5}}{D}{\inte}=0$.
  
  It remains to prove that $\operatorname{Ker}(d_1^{2,6})=0$.
  By using, essentially, the theorem on page 102 of \cite{Fu}, one computes from the fan $\Delta_5$ a basis $\{\tau_j^{(i_1,i_2)}\}$ for the cohomologies $\Homred{2}{D_{i_1, i_2}}$ and a basis $\{\tau_j^{(i_1)}\}$ for $\Homred{4}{D_{i_1}}$.
  Using the tools of intersection theory for toric varieties (see Chapter 5 in \cite{Fu}), one constructs the matrix of the homomorphism $d_1^{2,6}$: this is in Table \ref{tab-matrix-2-6-articolo}. 
  The details of the computation of such homomorphism can be found in Chapter 7 of \cite{Thesis-Martino}.
  This matrix has a zero dimensional kernel over $\inte$.
  \begin{table}[ht]
  \begin{center}
  \begin{tiny}
  \begin{tabular}{c|c|c|cc|c|c|c|cc|cccc}
  &$\tau_2^{5,6}$&$\tau_2^{2,4}$&$\tau_2^{2,3}$&$\tau_3^{2,3}$&$\tau_2^{1,3}$&$\tau_2^{1,4}$&$\tau_2^{1,6}$&$\tau_2^{1,5}$&$\tau_3^{1,5}$&$\tau_2^{1,2}$&$\tau_3^{1,2}$&$\tau_4^{1,2}$&$\tau_5^{1,2}$\\\hline
  $\tau_3^{(4)}$& 0&\M 1& 0& 0& 0&\M 1& 0 &0 &0 &0 &0 &0 &0\\\hline
  $\tau_3^{(6)}$&\M $\pm$1& 0& 0& 0& 0& 0&\M 1 &0 &0 &0 &0 &0 &0\\\hline
  $\tau_4^{(3)}$&0 & 0&\M 0&\M 1&\M 0& 0&0 &0 &0 &0 &0 &0 &0\\
  $\tau_5^{(3)}$&0 & 0&\M 1&\M 0&\M 1& 0&0 &0 &0 &0 &0 &0 &0\\\hline
  $\tau_4^{(5)}$&\M 0& 0& 0& 0& 0& 0& 0 &\M 0 &\M 1 &0 &0 &0 &0\\
  $\tau_5^{(5)}$&\M 1& 0& 0& 0& 0& 0& 0 &\M 1 &\M 0 &0 &0 &0 &0\\\hline

  $\tau_6^{(2)}$&0 &\M -2&\M 0&\M 0& 0& 0& 0 &0 &0 &\M 0 &\M 0 &\M 0 &\M 1\\
  $\tau_7^{(2)}$&0 &\M -1&\M 1&\M 0& 0& 0& 0 &0 &0 &\M 0 &\M 0 &\M 1 &\M 0\\
  $\tau_8^{(2)}$&0 &\M  0&\M 0&\M 0& 0& 0& 0 &0 &0 &\M 0 &\M 1 &\M 0 &\M 0\\
  $\tau_9^{(2)}$&0 &\M +1&\M 0&\M 1& 0& 0& 0 &0 &0 &\M 1 &\M 0 &\M 0 &\M 0\\\hline

  $\tau_5^{(1)}$	&0 &0 &0&0&\M 0&\M 0&\M -2 &\M 1 &\M -2 &\M 0 &\M 0 &\M 0 &\M 0\\
  $\tau_6^{(1)}$	&0 &0 &0&0&\M 0&\M 0&\M -3 &\M 0 &\M -3 &\M 0 &\M 0 &\M 0 &\M 0\\
  $\tau_{10}^{(1)}$&0 &0 &0&0&\M 0&\M -2&\M 0 &\M 0 &\M 0 &\M 0 &\M 0 &\M 0 &\M 1\\
  $\tau_{11}^{(1)}$&0 &0 &0&0&\M 1&\M -1&\M 2 &\M 0 &\M 2 &\M 0 &\M 0 &\M 1 &\M 0\\
  $\tau_{12}^{(1)}$&0 &0 &0&0&\M 0&\M 0&\M 4 &\M 0 &\M 4 &\M 0 &\M 1 &\M 0 &\M 0\\
  $\tau_{13}^{(1)}$&0 &0 &0&0&\M 0&\M 1&\M 3 &\M 0 &\M 3 &\M 1 &\M 0 &\M 0 &\M 0\\
  \end{tabular}
  \end{tiny}
  \end{center}
  \caption{The matrix of the differential $d_1^{2,6}$. We decorate columns and rows with the basis elements.}
  \label{tab-matrix-2-6-articolo}
  \end{table}
  %
\end{proof}

\begin{teo}\label{theo-cohomology-result-torsion}
    $\operatorname{tor}(\Hom{5}{\varns{5}}{\inte}) = 0$.
\end{teo}
\begin{proof}
Let us consider the long exact sequence
\[
  \dots \rightarrow \Hrel{*}{X_5}{E}{\inte} \rightarrow \Hom{*}{\varns{5}}{\inte} \rightarrow \Hom{*}{\open{5}}{\inte}\rightarrow \dots
\]
From Lemma \ref{lem-technical-lemma-betti-numebers}.\textbf{iii)}, we know that $\beta^{odd}(X_5)=0$ and hence $   \Hom{odd}{\varns{5}}{\inte}=\operatorname{tor}\left(\Hom{odd}{\varns{5}}{\inte}\right)$.
%
Theorem \ref{teo-cohomology-open-set-articolo} shows that $\Hom{odd}{\open{5}}{\inte}=0$ and hence we obtain
\[
 \dots \rightarrow \Hrel{5}{X_5}{E}{\inte} \rightarrow \operatorname{tor}(\Hom{5}{\varns{5}}{\inte}) \rightarrow 0.
\]
Theorem \ref{thm-rational-cohomology-exceptional-divisor-articolo} says that $\Hrel{5}{X_5}{E}{\inte}=0$ and one has $\operatorname{tor}(\Hom{5}{\varns{5}}{\inte})=0$.
\end{proof}

\noindent
\textbf{Acknowledgements}\\
I would like to thanks to Torsten Ekedahl for suggesting me this problem and for his advice.
I thank Anders Bj\"orner and Angelo Vistoli for the priceless help and great support.
I am very grateful to Fabio Tonini for his suggestions.

\noindent
Finally, my deep gratitude goes to Zinovy Reichstein for the comments on the paper. I have really appreciated the suggested improvements on Section 2. 
In particular, Proposition 2.2 has this form because of his comments.

The author is fully supported by the Swiss National Science Foundation,
grant number PP00P2\_150552/1.

\addcontentsline{toc}{section}{Bibliography}
\bibliographystyle{siam}
\bibliography{licentiate}

\vspace{0.5cm}

\noindent
 {\scshape Ivan Martino}\\
 {\scshape Department of Mathematics, University of Fribourg,\\ 1700 Fribourg, Switzerland}.\\
 {\itshape E-mail address}: \texttt{ivan.martino@unifr.ch}
\end{document}